\documentclass[12pt, a4paper]{amsproc}
\usepackage{amsmath,amssymb,anysize,times,float,enumerate}
\usepackage{listings}

\usepackage{hyperref} 
\hypersetup{citecolor=red, linkcolor=blue, colorlinks=true}

\newtheorem{thm}{Theorem}[section]

\newtheorem{prop}[thm]{Proposition}
\newtheorem{cor}[thm]{Corollary}

\theoremstyle{definition}

\newtheorem{function}{Function}[section]

\theoremstyle{remark}

\newcommand\Wr{{\rm wr\, }}

\DeclareMathOperator\PSL{{\rm PSL}}
\DeclareMathOperator\PGL{{\rm PGL}}
\DeclareMathOperator\AGL{{\rm AGL}}
\DeclareMathOperator\PGammaL{{\rm P}\Gamma{\rm L}}

\title[Covering Numbers of Small Symmetric Groups]{On the Covering Number of Small Symmetric Groups and Some Sporadic Simple Groups}

\author{Luise-Charlotte Kappe}
\address{Department of Mathematical Sciences, State University of New York at Binghamton, Binghamton, NY 13902-6000, USA.} 
\email{menger@math.binghamton.edu}

\author{Daniela Nikolova-Popova}
\address{Department of Mathematial Sciences, Florida Atlantic University, Boca Raton, FL 33431, USA.} 
\email{dpopova@fau.edu}

\author{Eric Swartz}
\address{ %
Centre for the Mathematics of Symmetry and Computation, School of Mathematics and Statistics,
The University of Western Australia,
35 Stirling Highway, Crawley, W.A. 6009, Australia.}
\email{eric.swartz@uwa.edu.au}



\begin{document}

\begin{abstract}
A set of proper subgroups is a covering for a group if its union is the whole group.  The minimal number of subgroups needed to cover $G$ is
called the covering number of $G$, denoted by $\sigma(G)$. Determining $\sigma(G)$ is an open problem for many non-solvable groups.  For
symmetric groups $S_n$, Mar{\'o}ti determined $\sigma(S_n)$ for odd $n$ with the exception of $n=9$ and gave estimates for $n$ even.  In this
paper we determine $\sigma(S_n)$ for $n = 8$, 9, 10 and 12.  In addition we find the covering number for the Mathieu group $M_{12}$ and
improve an estimate given by Holmes for the Janko group $J_1$. 
\end{abstract}

\maketitle

\section{Introduction} 
\label{S:1}

Let $G$ be a group and ${\mathcal A} = \{A_i \mid 1\leq i\leq n\}$ a collection of proper subgroups of $G$.  If $G = \displaystyle{\bigcup^n_{i=1}}A_i$,
then ${\mathcal A}$ is called a cover of $G$.  A cover is called irredundant if after the removal of any subgroup, the remaining subgroups do
not cover the group.  A cover of size $n$ is said to be minimal if no cover of $G$ has fewer than $n$ members.  According to J.H.E.\ Cohn
\cite{C}, the size of a minimal covering of $G$ is called the covering number, denoted by $\sigma(G)$.  By a result of B.H.\ Neumann \cite{N},
a group is the union of finitely many proper subgroups if and only if it has a finite noncyclic homomorphic image.  Thus it suffices to
restrict our attention to finite groups when determining covering numbers of groups.  

Determining the invariant $\sigma(G)$ of a group $G$ and finding the positive integers which can be covering numbers is the topic of ongoing
research.  It even predates Cohn's 1994 publication \cite{C}.  It is a simple exercise to show that no group is the union of two proper
subgroups. Already in 1926, Scorza \cite{Sc} proved that $\sigma(G)=3$ if and only if $G$ has a homomorphic image isomorphic to the Klein-Four
group, a result many times rediscovered over the years.  In \cite{Gr1}, Greco characterizes groups with $\sigma(G) = 4$ and in \cite{Gr2} and
\cite{Gr3} gives a partial characterization of groups with $\sigma(G) = 5$.  For further details we refer to the survey article by Serena
\cite{Se}, and for recent applications of this research see for instance \cite{B} and \cite{BEGHM}.  

In \cite{C}, Cohn conjectured that the covering number of any solvable group has the form $p^\alpha+1$, where $p$ is a prime and $\alpha$ a
positive integer, and for every integer of the form $p^\alpha+1$ he determined a solvable group with this covering number.  In \cite{T},
Tomkinson proves Cohn's conjecture and suggests that it might be of interest to investigate minimal covers of non-solvable and in particular
simple groups.  Bryce, Fedri and Serena \cite{BFS} started this investigation by determining the covering number for some linear groups such
as PSL$(2,q)$, PGL$(2,q)$ or GL$(2,q)$ after Cohn \cite{C} had already shown that $\sigma(A_5) = 10$ and $\sigma(S_5) = 16$. In \cite{L},
Lucido investigates Suzuki groups and determines their covering numbers.  For the sporadic groups, such as $M_{11}$, $M_{22}$, $M_{23}$, $Ly$
and $O'N$, the covering numbers 
are established in \cite{H} by Holmes and she gives estimates for those of $J_1$ and $M^cL$.  Some of the results in \cite{H} are
established with the help of GAP \cite{Ga}, a first in this context.

The covering numbers of symmetric and alternating groups were investigated by Mar{\'o}ti in \cite{M}.  For $n\neq 7$, 9, he shows that for the
alternating group $\sigma(A_n) \geq 2^{n-2}$ with equality if and only if $n$ is even but not divisible by 4.  For $n=7$ and 9 Mar{\'o}ti
establishes $\sigma(A_7) \leq 31$ and $\sigma(A_9) \geq 80$.  For the symmetric groups he proves that $\sigma(S_n) = 2^{n-1}$ if $n$ is odd
unless $n=9$ and $\sigma(S_n) \leq 2^{n-2}$ if $n$ is even.  It is a natural question to ask what are the exact covering numbers for
alternating and symmetric groups for those values of $n$ where Mar{\'o}ti only gives estimates.  In \cite{KR} and \cite{EMN} this was done for
alternating groups in case of small values of $n$.  As mentioned earlier,
Cohn \cite{C} already established $\sigma(A_5) = 10$.  In \cite{KR} it is shown that $\sigma(A_7) = 31$ and $\sigma(A_8) = 71$.  Furthermore,
Mar{\'o}ti's bound for $A_9$ is improved by establishing that $127 \leq \sigma(A_9) \leq 157$.  Recently, it was shown in \cite{EMN} that
$\sigma(A_9) = 157$.

The topic of this paper is to determine the covering numbers for symmetric groups of small degree and some sporadic simple groups.  We determine the covering numbers for $S_n$ in cases when $n = 8$, $9$, $10$, and $12$.  In
particular, we show $\sigma(S_9) = 256$, establishing that Mar{\'o}ti's result that $\sigma(S_n) = 2^{n-1}$ for odd $n$ holds without exceptions.  For $n
= 8$, $10$ and $12$ we have $\sigma(S_8) = 64$, $\sigma(S_{10}) = 221$, and $\sigma(S_{12}) = 761$, respectively.  We observe that Mar{\'o}ti
\cite{M} gave already 761 as an upper bound for $\sigma(S_{12})$.  Since we can use the same methods, we establish in
addition that the Mathieu group $M_{12}$ has covering number $208$ and improve the estimate given for the Janko group $J_1$ in \cite{H}.

Observing that $\sigma(S_4) = 4$ and $\sigma(S_6) = 13$ by \cite{AAS}, we know now the covering numbers of $S_n$ for all even $n \leq 12$ and
observe that in this range $\sigma(S_n) = 2^{n-2}$, Mar{\'o}ti's upper bound, is only taken if $n$ is a 2-power.  In the remaining cases we
have $\sigma(S_n) < 2^{n-2}$ and $\sigma(S_n) \sim \frac{1}{2} {n \choose {n/2}}$.  This suggests that perhaps the value for $\sigma(S_n)$ is less than Mar{\'o}ti's bound in case $n$ is not a
2-power.  Our current methods rely on explicit tables for the symmetric groups in question and computer calculation to carry out certain
optimizations.  There are limits to the size of the group on how far these methods can carry us and statements for general values of $n$ are extremely difficult and require entirely different methods than those used for small values of $n$.  This will become clearer when we discuss our methods in the following.

The methods employed here are an extension of those used in \cite{KR}.  In determining a minimal covering of a group we can restrict ourselves
to finding a minimal covering by maximal subgroups.  The conjugacy classes of subgroups for the groups in question can be found in GAP
\cite{Ga}.  To determine a minimal covering by maximal subgroups, it suffices to find a minimal covering of the conjugacy classes of
maximal cyclic subgroups by such subgroups of the group.  Already in \cite{H} this method is used to determine the covering numbers of
sporadic groups.  Here this method is adapted to the case of symmetric groups where the generators of maximal cyclic subgroups can easily be
identified by their cycle structure. 

The following notation is used for the disjoint cycle decomposition of a nontrivial permutation.  Let $m_1,m_2,\hdots,m_t\in {\mathbb N}$ with
$1 < m_1 < m_2 < \hdots < m_t$ and $k_1,\hdots,k_t\in {\mathbb N}$.  If $\alpha$ is a permutation with disjoint cycle decomposition of $k_i$
cycles of length $m_i$, $i = 1,\hdots,t$, then we denote the class of $\alpha$ by $(m_1^{k_1},\hdots,m_t^{k_t})$.  If $k_i = 1$, we just write
$m_i$ instead of $m_i^1$. As is customary, we suppress 1-cycles and the identity permutation is denoted by (1).  For example, the permutation
with disjoint cycle decomposition (12)(34)(5678) belongs to the class $(2^2,4)$.  In the case of symmetric groups all elements of a given cycle structure are contained in the subgroups of a conjugacy class of maximal subgroups and the elements with the respective cycle
structure are either partitioned into these subgroups or there exists an intersection between some of the subgroups of the conjugacy class.

For the groups $S_8$, $S_9$, $S_{10}$, and $M_{12}$, we provide two tables which are obtained with the help of GAP \cite{Ga}.  (For the group $S_{12}$, we provide only a list of maximal subgroup conjugacy classes and refer to previous work in \cite{M}.  For the group $J_1$, we refer to previous work in \cite{H}.)  The first table gives the
information on the conjugacy classes of maximal subgroups of the group, such as the isomorphism type and order of the class representative and
the size of each class.  The second table lists the order and cycle structure of each permutation generating a maximal cyclic subgroup as well
as the total of such elements in the group together with the distribution of these elements over the various conjugacy classes.  For each
conjugacy class we list how many of these elements are contained in a class representative.  If elements are partitioned over the
representatives, we indicate this with $P$, and if each element is contained in $k$ class representatives and each representative contains $s$
such elements, we indicate this with $s_k$.  For some of the groups it suffices to give the second table in abbreviated form.

For finding the covering number, the goal is to determine an irredundant covering and show that it is minimal.  If the elements of a certain
cycle structure are partitioned into the subgroups of a particular conjugacy class, it is not hard to find a minimal covering for such
elements.  The difficulty arises if the elements in question occur in several class representatives.  In this case we interpret the subgroups
and group elements as an incidence structure with the subgroup representatives as the sets and the group elements with the specific cycle
structure as elements.  This leads to a problem in linear optimization.  Here are some of the details.

Given two finite collections of objects, call them $U$ and $V$.  Call the objects in $V$ elements and the objects in $U$
sets.  Given an incidence structure between $U$ and $V$, that is for every $v$ in $V$ and every $u$ in $U$ we have
either $v$ incident with $u$ or $v$ not incident with $u$, $v\in u$ or $v\not\in u$ for short.  This relation can be
represented by a matrix $A = (a_{ij})$, the incidence matrix of $(V,U)$.  We label the columns of $A$ by the sets in
$U$ and the rows by the elements in $V$.  For $1\leq i \leq |V|$ and $1\leq j\leq |U|$ we set
\begin{equation*}
a_{ij} = \begin{cases} 1 &\text{if $v_i\in u_j$},\\
0 &\text{if $v_i\not\in u_j$.}\end{cases}
\end{equation*}
Let $W$ be a subcollection of $U$.  We define a column vector $x(W) = (x_1,\hdots,x_{|U|})^T$ as follows 
\begin{equation*}
x_j = \begin{cases} 1 &\text{if $u_j\in W$},\\
0 &\text{if $u_j\not\in W$.}\end{cases}
\end{equation*}
Let $Ax(W) = y(W) = (y_1,\hdots,y_{|V|})^T$, a column vector of length $|V|$ with coordinates $y_i \geq 0$.  If $y_i
= 0$, then $v_i \not\in \displaystyle{\bigcup_{u\in W}}u$ and if $y_i > 0$, then $v_i \in \displaystyle{\bigcup_{u\in W}}u$,
specifically $v_i$ is contained in exactly $y_i$ members of $W$.  If $y_i > 0$ for $i = 1,\hdots,|V|$, every $v_i\in
V$ is contained in at least one member of $W$ and we say $W$ covers $V$.  

In our interpretation the objects in $U$ are representatives of a certain conjugacy class of maximal subgroups and the objects in $V$ are
permutations with a certain cycle structure.  The goal is to find the minimal size of $|W|$ such that $W$ covers $V$.
If the objects in $U$ and $V$ can be suitably labeled, we can use a combinatorial argument to
find the optimal solution, e.g.,\ using the Erd\H{o}s-Ko-Rado Theorem \cite{E} as in the case of $S_{10}$.  Otherwise we have to resort to the
help of computers to find optimal solutions, e.g.,\ in the case of $S_9$, $M_{12}$ and $J_1$.  Roughly speaking, a system of linear inequalities with binary variables is prepared by GAP \cite{Ga} and the optimal solution is found with the help of Gurobi \cite{Gu}.  Naturally, this approach puts a limit on how
large our groups can be.  In addition, the structure of the cover heavily depends on the arithmetic nature of $n$.

\section{The Symmetric Group $S_8$}
\label{S:2}

The smallest symmetric group for which the covering number is not known is $S_8$.  Here we determine $\sigma(S_8)$ and show that it equals the upper bound given by Mar{\'o}ti
in \cite{M}.  

\begin{thm}  The covering number of $S_8$ is $64$.
\end{thm}

\begin{proof} 
First we will show that there exists an irredundant covering of $S_8$ by $64$ subgroups. As can be seen from Table 2.2, 
all odd permutations of the 
group generating maximal cyclic subgroups are contained either in $MS3$ or $MS6$. Thus the union of $MS3$ and $MS6$ 
contains all odd permutations in question. We observe that this union does not contain all even permutations 
generating maximal cyclic subgroups, e.g., the permutation with cycle structure $(3,5)$ is only contained in $MS1$ and 
$MS2$. Thus $MS1$, $MS3$, and $MS6$ cover all of $S_8$, and
\begin{equation*}
\sigma(S_8) \leq |MS1| + |MS3| + |MS6| = 64.
\end{equation*}
Let $\mathcal C$ be the union of $MS1$, $MS3$, and $MS6$, and define $\Pi$ to be the union of all elements with cycle structure $(8), (3,5)$, or $(2,3^2)$.   The elements of 
$\Pi$ are partitioned among the $64$ groups of $\mathcal{C}$, so $\mathcal{C}$ is an irredundant covering.  

% Note that $A_8$, the unique subgroup in MS1, contains all $2688$ elements with cycle structure $(3,5)$.  Each subgroup isomorphic to $S_3 \times S_5$ in MS2 contains $48$ elements with cycle structure $(3,5)$ and $40$ elements with cycle structure $(2,3^2)$ for a total of $88$ elements of $\Pi$.  Each subgroup isomorphic to $S_2 \times S_6$ in MS3 contains $40$ elements with cycle structure $(2,3^2)$, and no subgroup isomorphic to $S_7$ in MS4 contains any elements of $\Pi$.  The subgroups isomorphic to $S_2 \Wr S_4$ in MS5 each contain $48$ elements with cycle structure $(8)$ and $32$ elements with cycle structure $(2,3^2)$ for a total of $80$ elements of $\Pi$, and the subgroups isomorphic to $S_4 \Wr S_2$ in MS6 each contain $144$ elements of type $(8)$.  Finally, each subgroup isomorphic to $\PGL(2,7)$ in MS7 contains $84$ elements with cycle structure $(8)$.

It remains to be shown that ${\mathcal C}$ is a minimal covering.  Assume to the contrary that there exists a cover $\mathcal{B}$ of $S_8$ such that $\mathcal{B}$ contains 
fewer subgroups than $\mathcal{C}$.  Since $\mathcal{B}$ covers all the elements of $S_8$, it must cover all the elements of $\Pi$.  Moreover, $\mathcal{B}$ contains fewer 
subgroups than $\mathcal{C}$, so we may assume that $\mathcal{C} = (\mathcal{B} \cap \mathcal{C}) \cup \mathcal{C}'$ and $\mathcal{B} = (\mathcal{B} \cap \mathcal{C}) \cup 
\mathcal{B}'$, where $\mathcal{C}'$ is the set of subgroups in $\mathcal{C}$ but not in $\mathcal{B}$ and $\mathcal{B}'$ is the set of subgroups in $\mathcal{B}$ that are 
not in $\mathcal{C}$.  Since $|\mathcal{B}| < |\mathcal{C}|$, it must be that $|\mathcal{B}'| < |\mathcal{C}'|$.  This means that $\mathcal{B}'$ must cover some subset of 
elements of $\Pi$ more efficiently than does $\mathcal{C}'$.    

We will now show that $A_8$ is in ${\mathcal B} \cap {\mathcal C}$.  If $A_8\not\in {\mathcal B}\cap {\mathcal C}$, then the only other way to cover the elements of cycle
structure $(3,5)$ is by the 56 subgroups of $MS2$.  Since the most efficient way to cover the 8-cycles is by the 35 subgroups of $MS6$ and no maximal subgroup contains both elements of cycle structure $(3,5)$ and $8$-cycles, we have $|{\mathcal B}| \geq 56 + 35 >
|{\mathcal C}|$, a contradiction.  We conclude $A_8 \in {\mathcal B}\cap {\mathcal C}$.  

Define $\mathcal{C}_1'$ to be the set of subgroups of $\mathcal{C}'$ that are in $MS3$ and $\mathcal{C}_2'$ to be the set of subgroups of $\mathcal{C}'$ in $MS6$, and let 
$\Pi_1$ and $\Pi_2$, respectively, be the elements of $\Pi$ that are in subgroups of $\mathcal{C}_1'$ and ${\mathcal C}'_2$, respectively.  Note that $\Pi_1$ and $\Pi_2$ 
are disjoint since the elements of $\Pi$ are partitioned among the subgroups of $\mathcal{C}$.  As can be seen from examining Table 2.2, the maximal subgroups of $S_8$ 
that are not in $\mathcal{C}$ (i.e., those isomorphic to $S_3 \times S_5$, $S_7$, $\PGL(2,7)$, or $S_2 \Wr S_4$) contain at most $80$ elements of
$\Pi_1 \cup \Pi_2$, whereas a 
subgroup isomorphic to $S_4 \Wr S_2$ in $MS6$ contains $144$ elements of this set.  Since $\mathcal{C}$ partitions $\Pi_1\cup \Pi_2$, this means that if $\mathcal{C}_2'$ contains $n$ 
subgroups, then $\mathcal{B}'$ must contain at least $n+1$ subgroups to cover the elements in $\Pi_2$.  Since $|\mathcal{B}'| < |\mathcal{C}'|$, this means both that 
$\mathcal{C}_1'$ is nonempty and that some collection $\mathcal{B}_1'$ of $\mathcal{B}'$ covers the elements of $\Pi_1$ with fewer subgroups than $\mathcal{C}_1'$.  
However, $\Pi_1$ consists only of elements with cycle structure $(2,3^2)$, and each subgroup of $\mathcal{C}_1'$ contains exactly $40$ such elements.  For $\mathcal{B}_1'$ 
to be smaller than $\mathcal{C}_1'$, some subgroup of $S_8$ would have to contain more than $40$ elements with cycle structure $(2,3^2)$.  None does, which is a 
contradiction.  Therefore it follows that no such cover $\mathcal{B}$ can exist and $\mathcal{C}$ is a minimal covering of $S_8$. We conclude $\sigma(S_8) = 64$, as desired.

\end{proof}

\begin{center}
\begin{tabular}{c|c|c|c}
Label &Isomorphism Type &Group Order &Class Size\\
\hline
$MS1$ &$A_8$ &20160 &1\\
$MS2$ &$S_3\times S_5$ &720 &56\\
$MS3$ &$S_2\times S_6$ &1440 &28\\
$MS4$ &$S_7$ &5040 &8\\
$MS5$ &$S_2 \Wr S_4$ &384 &105\\
$MS6$ &$S_4 \Wr S_2$ &1152 &35\\
$MS7$ &$\PGL(2,7)$ &336 &120\\
\end{tabular}
\vskip 10pt

Table 2.1.  Conjugacy classes of maximal subgroups of $S_8$.
\end{center}
\vskip 20pt

\begin{center}
\begin{tabular}{c|c|c|c|c|c|c|c|c|c}
 & & &$MS1$ &$MS2$ &$MS3$ &$MS4$ &$MS5$ &$MS6$ &$MS7$\\
Order &C.S. &Size &(1) &(56) &(28) &(8) &(105) &(35) &(120)\\
\hline
ODD & & & & & & & & &\\
\hline
4 &$(2^2,4)$ &1260 &0 &0 &$90_2$ &0 &$36_3$ &$180_5$ &0\\
6 &(2,3) &1120 &0 &$100_5$ &$160_4$ &$420_3$ &0 &$96_3$ &0\\
6 &$(2,3^2)$ &1120 &0 &$40_2$ &$40,P$ &0 &$32_3$ &0 &0\\
6 &6 &13360 &0 &0 &$120,P$ &$840_2$ &$32,P$ &0 &$56_2$\\
8 &8 &5040 &0 &0 &0 &0 &$48,P$ &$144,P$ &$84_2$ \\
10 &(2,5) &4032 &0 &$72,P$ &$144,P$ &$504,P$ &0 &0 &0\\
12 &(3,4) &3360 &0 &$60,P$ &0 &$420,P$ &0 &$96,P$ &0\\
\hline
EVEN & & & & & & & & &\\
\hline
4 &(2,4) &2520 &$P$ &$90,P$ &$180_2$ &$630_2$ &$24,P$ &$72,P$ &0\\
6 &(2,6) &3360 &$P$ &0 &$120,P$ &0 &$32,P$ &$192_2$ &0\\
7 &7 &5760 &$P$ &0 &0 &$720,P$ &0 &0 &$48,P$\\
15 &(3,5) &2688 &$P$ &$48,P$ &0 &0 &0 &0 &0\\
\end{tabular}
\vskip 10pt

Table 2.2.  Inventory of elements generating maximal cyclic subgroups 

in $S_8$ across conjugacy classes of maximal subgroups.
\end{center}
\vskip 10pt

\section{The Symmetric Group $S_9$}
\label{S:3}

In this section we will determine the exact covering number of $S_9$, the case missing in \cite{M}, where the covering numbers for $S_n$ with $n$ odd
were determined with the exception of $n=9$.  

\begin{thm}\label{T:3.1}  The covering number of $S_9$ is $256$.
\end{thm}

This together with Theorem 1.1 in \cite{M} yields the following corollary.

\begin{cor} Let $n \geq 3$ be an odd integer.  Then $\sigma(S_n) = 2^{n-1}$.
\end{cor}

To prove the main result of this section, we need the following proposition.

\begin{prop}\label{P:3.2} The $84$ subgroups of $MS3$ form a minimal covering of the elements with cycle structure $(3,6)$ in $S_9$. 
\end{prop}

\begin{proof}  We prove this computationally with the help of the software GAP \cite{Ga} and Gurobi \cite{Gu}.  Using the GAP program as given in
Function \ref{F:8.1} for $G = S_9$ and the conjugacy classes $MS3$, $MS6$, and $MS7$ of maximal subgroups, we are setting up the equations readable by Gurobi for the
elements of type $(3,6)$.  The Gurobi output shows that a minimal covering of these elements consists of $84$ subgroups from $MS3$, $MS6$, and $MS7$.  Since the
elements with cycle structure $(3,6)$ are partitioned into the subgroups of MS3, these 84 subgroups constitute a minimal covering of these elements.
\end{proof}

We note that the GAP output addressed in the above proposition as well as an abbreviated Gurobi output of these calculations is given at the end of
Section \ref{S:8}.  For further details we refer to \url{http://www.math.binghamton.edu/menger/coverings/}.  Now we are ready to prove our theorem.

\begin{proof}[Proof of Theorem \ref{T:3.1}]  We will show first that there exists a covering of $S_9$ by $256$ subgroups.  As can be seen with
the help of GAP \cite{Ga}, the $9$-cycles in $S_9$ are only contained in $A_9$, the only subgroup in $MS1$.  Thus it suffices to show that the odd
permutations generating maximal cyclic subgroups can be covered by $255$ subgroups.

As can be seen from Table 3.2, listing the odd permutations generating maximal cyclic subgroups in $S_9$, the elements with cycle structure
$(4,5)$ and $(2,7)$ are only contained in the subgroups of $MS2$ and $MS4$, respectively.  Since these elements are partitioned into the subgroups
of the respective classes, the full classes have to be added to the covering.  As one can see from Table 3.2, the odd permutations generating
maximal cyclic subgroups not covered by the subgroups of $MS2$ and $MS4$ are the $8$-cycles and the elements with cycle structure $(3,6)$.  Thus
adding the subgroups of $MS3$ and $MS5$ to those of $MS1$, $MS2$ and $MS4$ provides a covering of $S_9$.  We conclude 
\begin{equation*}
\sigma(S_9) \leq |MS1| + |MS2| + |MS3| + |MS4| + |MS5| = 256.
\end{equation*}

It remains to be shown that any covering of $S_9$ contains at least $256$ subgroups.  As pointed out earlier, none of the subgroups of $MS1$, $MS2$
and $MS4$ can be omitted since the respective elements are partitioned into these subgroups.  The $8$-cycles are partitioned into the nine
subgroups of $MS5$ with $5040$ elements in each subgroup.  On the other hand, each such element is
contained in two subgroups of $MS7$ with $108$ $8$-cycles in each subgroup.  Obviously, replacing subgroups from $MS5$ by
those from $MS7$ increases the number of subgroups needed for covering these elements.  Hence the nine subgroups of MS5 constitute a
minimal covering of the $8$-cycles.

\begin{center}
\begin{tabular}{c|c|c|c}
Label &Isomorphism Type &Group Order &Class Size\\
\hline
$MS1$ &$A_9$ &18140 &1\\
$MS2$ &$S_4 \times S_5$ &2880 &126\\
$MS3$ &$S_3\times S_6$ &4320 &84\\
$MS4$ &$S_2 \times C_7$ &10080 &36\\
$MS5$ &$S_8$ &40320 &9\\
$MS6$ &$S_3 \Wr S_3$ &1296 &280\\
$MS7$ &$ \AGL(2,3)$ &432 &840\\
\end{tabular}
\vskip 10pt

Table 3.1.  Conjugacy classes of maximal subgroups of $S_9$.
\end{center}
\vskip 20pt

\begin{center}
\begin{tabular}{c|c|c|c|c|c|c|c|c}
Order &C.S &Size &$MS2$ &$MS3$ &$MS4$ &$MS5$ &$MS6$ &$MS7$\\
\hline
4 &$(2^2,4)$ &11340 &$180_2$ &$270_2$ &$630_2$ &$1260,P$ &$162_4$ &0\\
6 &(2,3) &2520 &$220_{11}$ &$270_9$ &$490_7$ &$1120_4$ &$36_4$ &0\\
6 &$(2,3^2)$ &10080 &$160_2$ &$360_3$ &$280,P$ &$1120,P$ &$36,P$ &0\\
6 &6 &10080 &0 &$120,P$ &$840_3$ &$3360_3$ &$36,P$ &$56_2$\\
6 &$(2^3,3)$ &2520 &$60_3$ &$30,P$ &$210_3$ &0 &$36_4$ &0\\
6 &(3,6) &20160 &0 &$240,P$ &0 &0 &$288_4$ &$72_3$\\
8 &8 &45360 &0 &0 &0 &$5040,P$ &0 &$108_2$\\
10 &(2,5) &18144 &$144,P$ &$432_2$ &$1008_2$ &$4032_2$ &0 &0\\
12 &(3,4) &15120 &360 &$180,P$ &$420,P$ &$3360_2$ &0 &0\\
14 &(2,7) &25920 &0 &0 &$720,P$ &0 &0 &0\\
20 &(4,5) &18144 &$144,P$ &0 &0 &0 &0 &0\\
\end{tabular}
\vskip 10pt

Table 3.2  Inventory of odd permutations generating maximal cyclic subgroups 

in $S_9$ across conjugacy classes of maximal subgroups.
\end{center}
\vskip 20pt
 
On the other hand, the elements with cycle structure $(3,6)$ are partitioned into the $84$ subgroups of $MS3$ with $240$ elements in each subgroup
and each such element is contained in four subgroups of $MS6$ with $288$ elements in each subgroup.  Thus potentially there could be an
arrangement that the elements with cycle structure $(3,6)$ in six subgroups of $MS3$ can be covered by five subgroups from $MS6$.  However, as
shown in Proposition \ref{P:3.2}, this is not the case, and the $84$ subgroups of $MS3$ constitute a minimal covering of these elements. Moreover, the only class of subgroups containing both $8$-cycles and elements with cycle structure $(3,6)$ is $MS7$.  However, each subgroup of $MS7$ contains a combined total of $180$ $8$-cycles and elements with cycle structure $(3,6)$, and so cannot possibly be a better cover than using $MS3$ and $MS5$, in each of which a subgroup covers at least $240$ such elements. 
We conclude $\sigma(S_9) > 255$ and thus $\sigma(S_9) = 256$.  
\end{proof}

\section{The Symmetric Group $S_{10}$}
\label{S:4}

In this section we determine the covering number of $S_{10}$.  It turns out to be less than the upper bound of $2^{10-2}$ given by Mar{\'o}ti in
\cite{M}.

\begin{thm}  The covering number of $S_{10}$ is $221$.
\end{thm}

Before we can prove Theorem 4.1, we have to establish some preparatory results involving combinatorics and incidence matrices leading to an application
of a result due to Erd\H{o}s, Ko and Rado \cite{E} (see also Theorem 5.1.2 in \cite{A}).

\begin{thm}{\cite{E}}  
Let $A_1,\hdots,A_m$ be $m$ $k$-subsets of an $n$-set $S$, $k\leq \frac{1}{2}n$, which are pairwise nondisjoint. Then
$m \leq \binom{n-1}{k-1}$.  The upper bound for $m$ is best possible.  It is attained when the $A_i$ are precisely the
$k$-subsets of $S$ which contain a chosen fixed element of $S$.
\end{thm}

In our application we consider the following incidence structure.  We let 
\begin{equation*}
U = \{(k_1,k_2,k_3) : k_1,k_2,k_3\in \{0,1,\hdots,9\}\ \text{ and }\ k_1 < k_2 < k_3)\}
\end{equation*}
and 
\begin{equation*}
V = \{(u,u') : u,u'\in U\ \text{ with }\ u\cap u' = \emptyset\}.
\end{equation*}
We define an incidence relation between $U$ and $V$ as follows.  For $v = (u,u') \in V$ we say $v\in u_j$ if $u =
u_j$ or $u' = u_j$, and $v\not\in u_j$ otherwise.  For this choice of $U$ and $V$ we make the following claim.

\begin{prop} 
Let $U$, $V$ and the incidence relation between them defined as above.  Then there exists a subcollection $W^*$ of $U$
with $|W^*| = 84$ which covers $V$ and every subcollection $W$ of $U$ with $|W| < |W^*|$ does not cover $V$.
Specifically, $W^*$ can be chosen as $U-D$, where 
\begin{equation*}
D = \{(0,k_2,k_3) : k_2,k_3 \in \{1,2,\hdots,9\}, k_2 < k_3\}.
\end{equation*}
\end{prop}

\begin{proof}  We have $|U| = 120$ and $|V| = 2100$.  Thus the incidence matrix $A$ of $U$
and $V$ is a $2100 \times 120$ matrix with exactly two entries equal to 1 in each row, since\\ $u_j =
((k_1,k_2,k_3) , (k'_1,k'_2,k'_3)) \in v_i$ if and only if $(k_1,k_2,k_3) = u_j$ or $(k'_1,k'_2,k'_3) = u_j$.  With
$x(U) = (1,\hdots,1)^T$ we have $Ax(U) = (2,\hdots,2)^T$.  Let $u,u' \in U$ with $u\cap u' = \emptyset$ and let $X =
U-\{u,u'\}$.  Then $y(X)$ contains a zero entry and $v = (u,u')$ is not covered by $X$. On the other hand, removing
any subset $\{u_1,\hdots,u_t\}$ of $U$ with pairwise non-trivial intersection, i.e.\ $u_i \cap u_j \neq \emptyset$, then
for $X = U-\{u_1,\hdots,u_t\}$ the vector $y(X)$ has all non-zero entries.  The largest number of sets we can remove from $U$ has
the cardinality of a maximal set with pairwise non-trivial intersection. Applying Theorem 4.2 with $n = 10$ and $k =
3$, we obtain $\binom{9}{2} = 36$ for the cardinality of such a set.  Specifically, $D = \{(0,k_1,k_2) : k_1 < k_2,
k_1,k_2 \in \{1,\hdots,9\}\}$ is such a set.  Let $W^* = U-D$.  Then $y(W^*)$ has all entries $> 0$. On the other
hand, for any set $W$ with $|W| < |W^*|$ there exist $u,u' \in \overline{W}$, the complement of $W$ in $U$, such that 
$u\cap u' =
\emptyset$ and thus $y(W)$ has at least one zero entry.
\end{proof}

The following corollary establishes a minimal covering of the elements of type $(3^2,4)$ by certain subgroups from
$MS3$ (see Table 4.2).  Since these subgroups are isomorphic to $S_3 \times S_7$, we can label them by the letters fixed by the
respective $S_7$.  Hence we have 
\begin{equation*}
MS3 = \{H(k_1,k_2,k_3) : k_1,k_2,k_3 \in \{0,1,\hdots,9\},k_1 < k_2 < k_3\}.
\end{equation*}

\begin{cor}
Let ${\mathcal D} = \{H(0,k_2,k_3) : k_2,k_3\in\{1,2,\hdots,9\},k_2 < k_3\}$.  Then $\overline{\mathcal D} = MS3 -
{\mathcal D}$, the complement of ${\mathcal D}$ in $MS3$, is a minimal covering of the elements of type $(3^2,4)$
in $S_{10}$.  
\end{cor}

\begin{proof}  By Table 4.2, there are $50400$ elements of type $(3^2,4)$ in $S_{10}$.  Each $H(k_1,k_2,k_3)\in MS3$
contains 840 such elements and each element of type $(3^2,4)$ is in exactly two subgroups of $MS3$.  There are
exactly six cyclic subgroups generated by elements of type $(3^2,4)$ in
the intersection of $H(u)$ and $H(u')$ with $u = (k_1,k_2,k_3)$ and $u' = (k'_1,k'_2,k'_3)$ with each such cyclic
subgroup of order 12 containing four elements of type $(3^2,4)$.  Thus any two members $H(u)$ and $H(u')$ of $MS3$
with $u\cap u' = \emptyset$ share exactly 24 elements of type $(3^2,4)$.  The six cyclic subgroups of order 12 can be
represented as $\langle t \cdot c_{4i}\rangle$, $i = 1,2,3$, where $t = u\cdot u'$ or $u^{-1}\cdot u'$ and $c_{4i}$ is a
4-cycle in $\{j_1,j_2,j_3,j_4\}$, the complement of $\{k_1,k_2,k_3,k'_1,k'_2,k'_3\}$ in $\{0,1,\hdots,9\}$, specifically
$c_{41} = (j_1,j_2,j_3,j_4)$, $c_{42} = (j_1,j_3,j_2,j_4)$ and $c_{43}= (j_1,j_2,j_4,j_3)$.  For $u\cap u' = \emptyset$ we
consider 
\begin{equation*}
T(u,u') = \{g\in \langle t\cdot c_{4i}\rangle : t=u\cdot u'\ \text{ or }\  u^{-1}u'; i = 1,2,3;|g| = 12\}.
\end{equation*}
We have $|T(u,u')| = 24$.  The $50400$ elements of type $(3^2,4)$ are partitioned into the $2100$ equivalence classes
$T(u,u')$.  Identifying $MS3$ with $U$ of Proposition 4.3 and setting $V = \{T(u,u');u\cap u' = \emptyset\}$, we have the
same incidence structure as in Proposition 4.3 and the conclusion of the corollary follows immediately.
\end{proof}

Now we are ready to prove Theorem 4.1.

\begin{proof}[Proof of Theorem 4.1]  By Table 4.2 and Corollary 4.4 we see that\break $MS1 \cup MS5 \cup MS7\cup 
\bar{\mathcal D} =
S_{10}$, since all permutations generating maximal cyclic subgroups in $S_{10}$ are contained in the union of these
subgroups.  Hence 
\begin{equation*}
\sigma(S_{10}) \leq |MS1\cup MS5 \cup MS7 \cup \bar{\mathcal D}| = 1 + 10 + 126 + 84 = 221.
\end{equation*}

It remains to be shown that the covering obtained by the $221$ subgroups is minimal, i.e.\ $\sigma(S_{10}) \geq 221$.   First we will show that the $126$
subgroups of $MS7$ and any nine subgroups of $MS5$ constitute a minimal covering of the odd permutations generating maximal cyclic subgroups and not
involved in $MS3$.  We observe that the $10$-cycles are partitioned into the three conjugacy
classes $MS6$, $MS7$ and $MS8$.  Since $MS7$ contains only $126$ subgroups versus the $945$ and $2520$ subgroups,
respectively, of the other two classes, the $126$ subgroups of $MS7$ constitute a minimal covering of the $10$-cycles in
$S_{10}$.  Next we will show that any nine subgroups of $MS5$ are a minimal covering of the
$8$-cycles in $S_{10}$.  It is obvious that a minimal covering of the 8-cycles cannot be obtained from using subgroups from $MS4$, $MS6$
or $MS8$.  Let $MS5 = \{S^{(i)}_9;i=0,\hdots,9\}$ with $S^{(i)}_9 \cong S_9$ and fixed point $i$.  Any $8$-cycle 
in $S_{10}$ has two fixpoints, say $i_1$ and $i_2$.  After removing $S^{(i)}_9$, all $8$-cycles in $S_{10}$ are still
covered by the remaining subgroups in $MS5$.  Removing an additional subgroup from $MS5$, say $S^{(i_2)}_9$, leaves
those $8$-cycles with fixed points $i_1$ and $i_2$ uncovered.  Thus any nine subgroups of $MS5$ constitute a minimal
covering of the $8$-cycles in $S_{10}$.  It can be seen now from Table 4.2 that all odd permutations generating maximal cyclic subgroups and not
involved in $MS3$ are covered by the subgroups of $MS7$ and any nine subgroups of $MS5$.  

By Corollary 4.4, the $84$ subgroups of $\overline {\mathcal D}$ constitute a minimal covering of the elements of type $(3^2,4)$.  We observe now that the $219$
subgroups of $MS7 \cup \overline {\mathcal D} \cup {\mathcal C}_0$, where ${\mathcal C}_0 = \{S^{(i)}_9;i=1,2,\hdots,9\}$ constitute a minimal covering of the
elements of type $(3^2,4)$, the $10$-cycles and the $8$-cycles, since these elements are mutually not contained in the respective subgroups covering the
other types of elements.  However, not all odd permutations generating maximal cyclic subgroups are contained in ${MS7} \cup \overline{\mathcal D} \cup {\mathcal
C}_0$, specifically the elements of type $(2,7)$ and $(3,6)$ with fixpoint 0.  Adding $S^{(10)}_9$ to the covering, we obtain that the $220$ subgroups of
$MS7\cup MS5\cup \overline {\mathcal D}$ minimally cover the odd permutations of $S_{10}$ generating maximal cyclic subgroups.  A look at Table 4.2 shows that
the only even permutations generating maximal cyclic subgroups and not contained in $MS5$ and $MS7$ are the elements of type $(3,7)$.  They are
partitioned into $MS3$ and $MS1$.  Adding the single subgroup of $MS1$, which is isomorphic to $A_{10}$, to the cover yields $\sigma(S_{10}) \geq |MS1
\cup MS5\cup MS7\cup \overline {\mathcal D}| = 221$.  This together with the above leads to $\sigma(S_{10}) = 221$.  
\end{proof}

\begin{center}
\begin{tabular}{c|c|c|c}
Label &Isomorphism Type &Group Order &Class Size\\
\hline
$MS1$ &$A_{10}$ &1814400 &1\\
$MS2$ &$S_4\times S_6$ &17280 &210\\
$MS3$ &$S_3 \times S_7$ &30240 &120\\
$MS4$ &$S_2\times S_8$ &80640 &45\\
$MS5$ &$S_9$ &362880 &10\\
$MS6$ &$S_2 \Wr S_5$ &3840 &945\\
$MS7$ &$S_5 \Wr S_2$ &28800 &126\\
$MS8$ &$\text{P}\Gamma\text{L}(2,9)$ &1440 &2520\\
\end{tabular}
\vskip 10pt

Table 4.1.  Conjugacy classes of maximal subgroups of $S_{10}$.
\end{center}
\vskip 10pt

{\fontsize{8}{10pt}\selectfont
\begin{center}
\begin{tabular}{c|c|c|c|c|c|c}
Order &C.S. &Size &$MS1$ &$MS2$ &$MS3$ &$MS4$\\
\hline
ODD & & & & & & \\
\hline
4 &$(2^2,4)$ &56700 &0 &$1080_4$ &$1890_4$ &$3780_3$\\
4 &$(2,4^2)$ &56700 &0 &$540_2$ &0 &$1260,P$\\
6 &$(2^3,3)$ &25200 &0 &$480_4$ &$840_4$ &$1680_3$\\
6 &$(2,3^2)$ &50400 &0 &$1200_5$ &$1680_4$ &$2240_2$\\
6 &$(2^2,6)$ &75600 &0 &$360,P$ &0 &$3360_2$\\
6 &$(3,6)$ &201600 &0 &$960,P$ &$1680,P$ &0\\
8 &8 &226800 &0 &0 &0 &$5040,P$\\
10 &10 &362880 &0 &0 &0 &0\\
12 &$(3^2,4)$ &50400 &0 &$240,P$ &$840_2$ &0\\
14 &$(2,7)$ &259200 &0 &0 &$2160,P$ &$5760,P$\\
20 &$(4,5)$ &181440 &0 &$964,P$ &0 &0\\
30 &$(2,3,5)$ &120960 &0 &0 &$1008,P$ &$2688,P$\\
\hline
EVEN & & & & & & \\
\hline
6 &$(2,6)$ &151200 &$P$ &$720,P$ &$72520_2$ &$6720_2$\\
8 &$(8,2)$ &226800 &$P$ &0 &0 &$5040,P$\\
9 &9 &403200 &$P$ &0 &0 &0\\
12 &$(4,6)$ &151200 &$P$ &$720,P$ &0 &0\\
12 &$(2,3,4)$ &151200 &$P$ &$1440_2$ &$2520_2$ &$3360,P$\\
21 &$(3,7)$ &172800 &$P$ &0 &$1440,P$ &0\\
\end{tabular}
\vskip 10pt

\begin{tabular}{c|c|c|c|c|c|c}
Order &C.S. &Size &$MS5$ &$MS6$ &$MS7$ &$MS8$\\
\hline
ODD & & & & & &\\
\hline
4 &$(2^2,4)$ &56700 &$11340_2$ &$180_3$ &$900_2$ &0\\
4 &$(2,4^2)$ &56700 &0 &$300_5$ &$1800_4$ &$90_4$\\
6 &$(2^3,3)$ &25200 &$2520,P$ &0 &$600_3$ &0\\
6 &$(2,3^2)$ &50400 &$10080_2$ &$160_3$ &$800_2$ &0\\
6 &$(2^2,6)$ &75600 &0 &$240_3$ &$2400_4$ &0\\
6 &$(3,6)$ &201600 &$20160,P$ &0 &0 &$240_3$\\
8 &8 &226800 &45360 &$240,P$ &0 &$180_2$\\
10 &10 &362880 &0 &$384,P$ &$2880,P$ &$144,P$\\
12 &$(3^2,4)$ &50400 &0 &$160_3$ &0 &0\\
14 &$(2,7)$ &259200 &$25920,P$ &0 &0 &0\\
20 &$(4,5)$ &181440 &$18144,P$ &0 &$1440,P$ &0\\
30 &$(2,3,5)$ &120960 &0 &0 &$960,P$ &0\\
\hline
EVEN & & & & & & \\
\hline
6 &$(2,6)$ &151200 &$30240_2$ &$160,P$ &0 &0\\
8 &$(8,2)$ &226800 &0 &$240,P$ &$3600_2$ &$180_2$\\
9 &9 &403200 &$40320,P$ &0 &0 &0\\
12 &$(4,6)$ &151200 &0 &$160,P$ &$2400_2$ &0\\
12 &$(2,3,4)$ &151200 &$15120,P$ &0 &$1200,P$ &0\\
21 &$(3,7)$ &172800 &0 &0 &0 &0\\
\end{tabular}
\vskip 10pt

Table 4.2.  Inventory of elements generating maximal cyclic subgroups 

in $S_{10}$ across conjugacy classes of maximal subgroups.
\end{center}
}

\section{The Symmetric Group $S_{12}$}
\label{S:5}

In \cite{M}, Mar{\'o}ti gives an upper bound for the covering number of $S_{12}$, which is lower than the general upper bound given there.  We will show here that this bound
is indeed the covering number of $S_{12}$.  

\begin{thm}\label{thm:S12} The covering number of $S_{12}$ is $761$.
\end{thm}

\begin{proof}
As noted by Mar{\'o}ti \cite[p. 104]{M}, the covering number of $S_{12}$ is at most 761, since $S_{12}$ may be written as the union of all subgroups conjugate to 
$S_6 \Wr S_2$, $S_{11} \times S_{1}$, $S_{10} \times S_2$, $S_9 \times S_3$, and $A_{12}$, which correspond to the classes $MS5$, $MS2$, $MS3$, $MS4$, and $MS1$, respectively, of 
Table 5.1.  Indeed, we will show that this is in fact a minimal cover of $S_{12}$ by demonstrating that there is a particular class of maximal cyclic subgroups that is minimally covered by one of these five classes. 

First, we examine the elements with cycle structure $(12)$, i.e., the $12$-cycles of $S_{12}$.  It is not hard to see that the classes of maximal subgroups containing 
$12$-cycles are all imprimitive subgroups in the classes $MS5$, $MS8$, $MS9$, and $MS10$ (a $12$-cycle preserves such an imprimitive decomposition of twelve elements), and also the 
subgroups of class $MS11$.  Moreover, it is easy to see that the $12$-cycles must be partitioned in each of the classes $MS5$, $MS8$, $MS9$, and $MS10$, respectively, since a 
$12$-cycle stabilizes a unique imprimitive decomposition of twelve elements.  Since the $12$-cycles are partitioned among the subgroups in classes $MS5$, $MS8$, $MS9$, and 
$MS10$, respectively, and $MS5$ has the fewest number of subgroups, removing $n$ subgroups from $MS5$ from the cover would require at least $n+1$ replacements from the other 
classes.  On the other hand, the $12$-cycles are not partitioned in $MS11$, and simple computation using GAP \cite{Ga} shows that each such subgroup contains $220$ different 
$12$-cycles.  Removing even one subgroup from $MS5$, which contains $86400$ different $12$-cycles, would require at least $\lceil 86400/220 \rceil = 393$ different subgroups from $MS11$ to replace it.  Since there are only $462$ total subgroups in $MS5$, it is easy to see that the unique minimal covering of the maximal cyclic subgroups generated by $12$-cycles uses the $462$ subgroups from $MS5$.

Next, we examine the elements with cycle structure $(3,4,5)$.  These elements are only contained in the classes $MS4$, $MS6$, and $MS7$.  Since elements with this cycle structure 
preserve a unique intransitive partition of twelve elements into one set of size nine (by the $5$-cycle and the $4$-cycle) and one set of size three (by the $3$-cycle), 
the elements with cycle structure $(3,4,5)$ are partitioned among the subgroups of $MS4$.  Similar reasoning shows that these elements are also partitioned in $MS6$ and $MS7$, 
respectively.  Arguing as we did for the $12$-cycles above, we see that the unique minimal covering of these elements uses the $220$ subgroups from the class $MS4$. 

We now examine the elements with cycle structure $(2,5^2)$.  These elements are only contained in the classes $MS3$ and $MS7$.  While these elements are partitioned in class $MS3$, they are not partitioned in class $MS7$.  On the other hand, each subgroup of $MS3$ contains $72576$ elements with cycle structure $(2,5^2)$, whereas each subgroup in class $MS7$ contains $12096$ elements with this cycle structure.  Hence removing any collection of subgroups from $MS3$ requires at least $72576/12096 = 6$ times as many subgroups from $MS7$, and the unique minimal covering of these elements uses the $66$ subgroups from $MS3$.

Looking at the elements with cycle structure $(4,7)$, we see that these are contained only in subgroups of the classes $MS2$, $MS6$, and $MS7$.  As with elements examined above, these are partitioned among these three classes, so we see that the unique minimal covering of these elements uses the $12$ subgroups from $MS2$.

Finally, we examine the elements with cycle structure $(5,7)$.  These elements are contained in the subgroups of the classes $MS1$ and $MS7$, and they are partitioned among 
the subgroups of $MS7$.  Since $MS1$ only contains one subgroup (the alternating group $A_{12}$), the unique minimal cover of the elements with cycle structure $(5,7)$ uses the 
single subgroup from $MS1$.

It only remains to be shown now that no collection of subgroups from $MS7$ is a more efficient cover of some elements with cycle structure $(3,4,5)$, $(2, 5^2)$, $(4,7)$, 
and $(5,7)$ collectively than those listed above.  First, in order to cover all the elements with cycle structure $(5,7)$ which are contained in $A_{12}$, the single 
subgroup in $MS1$, we would need all $792$ subgroups of $MS7$, which is larger than our bound of $761$.  To cover the elements that are lost when a single subgroup of 
$MS2$ isomorphic to $S_{11}$ is removed, $330$ subgroups of $MS7$ are required.  However, the $462$ subgroups of $MS5$ are still needed, so this is a total of $792$ subgroups, 
more than our current bound of $761$.  Hence we need only consider the elements with cycle structure $(3,4,5)$ and $(2,5^2)$.  However, one subgroup of $MS7$ contains $10080$ elements with cycle structure $(3,4,5)$ and $12096$ elements with cycle structure $(2,5^2)$ for a total of $22176$ elements of one of these two types, whereas one subgroup of $MS3$ contains $72576$ elements with cycle structure $(2,5^2)$, and one subgroup of $MS4$ contains $36288$ elements with cycle structure $(3,4,5)$.  Since the elements are partitioned across $MS3$ and $MS4$, this shows that no collection of subgroups of $MS7$ can possibly be a more efficient cover. 

Putting this all together, we see that each of the classes $MS1$, $MS2$, $MS3$, $MS4$, and $MS5$ is necessary in a minimal cover;  on the other hand, these five classes together 
form a cover.  Therefore, these five classes together form the unique minimal cover of the elements of $S_{12}$, and the covering number of $S_{12}$ is $761$.
\end{proof}

\begin{center}
\begin{tabular}{c|c|c|c}
Label &Isomorphism Type &Group Order &Class Size\\
\hline
$MS1$ & $A_{12}$ & 239500800 & 1\\
$MS2$ & $S_{11} (\times S_{1})$ & 39916800 & 12\\
$MS3$ & $S_{10} \times S_2$ & 7257600 & 66\\
$MS4$ & $S_9 \times S_3$ & 2177280 & 220\\
$MS5$ & $S_6 \Wr S_2$ & 1036800 & 462\\
$MS6$ & $S_8 \times S_4$ & 967680 & 495\\
$MS7$ & $S_7 \times S_5$ & 604800 & 792\\
$MS8$ & $S_4 \Wr S_3$ & 82944 & 5775\\
$MS9$ & $S_2 \Wr S_6$ & 46080 & 10395\\
$MS10$ & $S_3 \Wr S_4$ & 31104 & 15400\\
$MS11$ & $\PGL(2,11) $ & 1320 & 362880\\
\end{tabular}
\vskip 10pt

Table 5.1.  Conjugacy classes of maximal subgroups of $S_{12}$.
\end{center}

\section{The Mathieu Group $M_{12}$}
\label{S:6}

Only as recently as 2010, it was shown by Holmes and Mar{\'o}ti in \cite{HM} that for the Mathieu group $M_{12}$ we have $131 \leq \sigma(M_{12}) \leq 222$.  Here we will
determine the exact covering number of $M_{12}$.

\begin{thm}\label{T:6.1}  The covering number of $M_{12}$ is $208$.
\end{thm}

Before we can prove this theorem, we need a proposition which gives a minimal covering for the elements with cycle structure $(6,6)$.  (We note that $M_{12}$ is represented here as a permutation group
embedded into $S_{12}$.)  In fact, the minimal cover found contains subgroups from three different conjugacy classes of subgroups.  This seems to be a first in this context and explains why the covering number for the group $M_{12}$ was not determined any earlier despite its
relatively small order.  The use of GAP and Gurobi led to this breakthrough.

\begin{prop}\label{P:6.2}  There exists a covering of the elements with cycle structure $(6,6)$ in $M_{12}$ by $130$ subgroups, and
this covering is minimal.  This covering is made up of $120$ subgroups isomorphic to $\PSL(2,11)$ from $MS5$, eight subgroups isomorphic to $C_2\times S_5$ from $MS8$, and two
subgroups isomorphic to $(C_4 \times C_4) : D_6$ in $MS10$.
\end{prop}

\begin{proof} Using the GAP \cite{Ga} program listed in Function \ref{F:8.1} 
for the elements with cycle structure $(6,6)$ in $M_{12}$ and the appropriate maximal 
subgroups of $M_{12}$, Gurobi \cite{Gu} finds that there exists a covering of the elements with cycle structure $(6,6)$ by 130 subgroups in $MS5$, $MS8$, 
$MS10$, and that this covering is minimal.  
\end{proof}

A list of generators for the subgroups of $M_{12}$ contained in this can be found on line at \url{http://www.math.binghamton.edu/menger/coverings/}.  Now we are ready to prove our theorem.

\begin{center}
\begin{tabular}{c|c|c|c}
Label &Isomorphism Type &Group Order &Class Size\\
\hline
$MS1$ &$M_{11}$ &7920 &12\\
$MS2$ &$M_{11}$ &7920 &12\\
$MS3$ &$\PGammaL(2,9)$ &1440 &66\\
$MS4$ &$\PGammaL(2,9)$ &1440 &66\\
$MS5$ &PSL(2,11) &660 &144\\
$MS6$ &$(C_3\times C_3) : (C_2 \times S_4)$ &432 &220\\
$MS7$ &$(C_3\times C_3) : (C_2 \times S_4)$ &432 &220\\
$MS8$ &$S_5 \times C_2$ &240 &396\\
$MS9$ &$2^{1+4} : S_3$ &192 &495\\
$MS10$ &$(C_4 \times C_4) : D_{12}$ &192 &495\\
$MS11$ &$A_4 \times S_3$ &72 &1320\\
\end{tabular}
\vskip 10pt

Table 6.1.  Conjugacy classes of maximal subgroups of $M_{12}$.
\end{center}
\vskip 20pt

\begin{center}
\begin{tabular}{c|c|c|c|c|c|c|c}
Order &C.S. &Size &$MS1$ &$MS2$ &$MS3$ &$MS4$ &$MS5$\\
 & & &(12) &(12) &(66) &(66) &(144)\\
\hline
6 &(2,3,6) &15840 &$1320,P$ &$1320,P$ &$240,P$ &$240,P$ &0\\
6 &(6,6) &7920 &0 &0 &0 &0 &$110_2$\\
8 &(8,2) &11880 &0 &1980 &0 &360 &0\\
8 &(4,8) &11880 &1980 &0 &360 &0 &0\\
10 &(2,10) &9504 &0 &0 &$144,P$ &$144,P$ &0\\
11 &(11) &17280 &$1440,P$ &$1440,P$ &0 &0 &$120,P$\\
\end{tabular}
\vskip 10pt

\begin{tabular}{c|c|c|c|c|c|c|c|c}
Order &C.S. &Size &$MS6$ &$MS7$ &$MS8$ &$MS9$ &$MS10$ &$MS11$\\
 & &  &(220) &(220) &(396) &(495) &(495) &(1320)\\
\hline
6 &(2,3,6) &15840 &144 &144 &0 &$32,P$ &0 &24\\
6 &(6,6) &7920 &0 &0 &60 &0 &32 &$6,P$\\
8 &(8,2) &11880 &0 &108 &0 &$24,P$ &$24,P$ &0\\
8 &(4,8) &11880 &108 &0 &0 &$24,P$ &$24,P$ &0\\
10 &(2,10) &9504 &0 &0 &$24,P$ &0 &0 &0\\
11 &(11) &17280 &0 &0 &0 &0 &0 &0\\
\end{tabular}
\vskip 10pt

Table 6.2.  Inventory of elements generating maximal cyclic subgroups  

in $M_{12}$ across conjugacy classes of maximal subgroups.
\end{center}

\begin{proof}[Proof of Theorem \ref{T:6.1}]  It can be easily seen from Table 6.2 that the subgroups in $MS1$ and MS4 cover all elements in
$M_{12}$ generating maximal cyclic subgroups with the exception of elements of cycle structure $(6,6)$.  By Proposition \ref{P:6.2} there
exists a covering of the elements of cycle structure $(6,6)$ by $130$ subgroups in $MS5$, $MS8$, and $MS10$.  Thus 
\begin{equation*}
\sigma(M_{12}) \leq |MS1| + |MS4| + 130 = 12 + 66 + 130 = 208.
\end{equation*}

It remains to be shown that any covering of $M_{12}$ contains at least 208 subgroups.  As can be seen from Table 6.2, a covering of the
$11$-cycles needs to contain at least $12$ subgroups of $MS1$ or $MS2$.  Similarly, a covering of the elements of cycle structure $(2,10)$ needs to
contain the 66 subgroups of $MS3$ or $MS4$.  Since the covering of the elements with cycle structure $(6,6)$ by the $130$ subgroups from $MS5$, $MS8$, and $MS10$ is
minimal by Proposition \ref{P:6.2}, it follows that $\sigma(M_{12}) \geq 12 + 66 + 130 = 208$.  We conclude $\sigma(M_{12}) = 208$.  
\end{proof}

\section{The Janko Group $J_1$}
\label{S:7}

In \cite{H} it was shown by Holmes that $5165 \leq \sigma(J_1) \leq 5415$ for the Janko group $J_1$.  Using similar methods employed in this paper for
$S_9$ and $M_{12}$, we were able to improve these bounds.  It should be noted here that longer computation times on more powerful machines would likely
improve these bounds.

To better utilize the results from \cite{H}, we will follow Holmes and use notation from the Atlas \cite{CCNPW} rather than representing the groups as
a permutation group as done in the previous cases.  Recall that conjugacy classes of elements are named by the orders of their elements and a capital
letter.  They are written in descending order of centralizer size.  Here is our improved estimate for $\sigma(J_1)$.  

\begin{thm}
For the covering number of the Janko group $J_1$ we have $5281 \leq \sigma(J_1) \leq 5414$.  
\end{thm}

\begin{proof}
In \cite{H} it is determined that all $1540$ maximal subgroups isomorphic to $C_{19} : C_6$ and all $2926$ maximal subgroups isomorphic to $S_3 \times D_{10}$
are needed in a minimal covering.  The only remaining elements generating maximal cyclic subgroups that need to be covered are those of type $11A$ and
$7A$.  Holmes shows in \cite{H} that only maximal subgroups isomorphic to $\PSL(2,11)$ are needed to cover all elements of type $11A$, and also only
maximal subgroups isomorphic to $C^3_2 : C_7 : C_3$ are needed to cover elements of type $7A$.  Using the GAP program \cite{Ga} as given in Function
\ref{F:8.1} for $G = J_1$ and the maximal subgroups isomorphic to $\PSL(2,11)$, 
we are setting up the equations readable by Gurobi \cite{Gu} for the elements
of type $11A$.  The Gurobi output then tells us that a minimal covering of the elements of this type consists of at least $186$ and at
most $196$ subgroups isomorphic to $\PSL(2,11)$. Similarly, preparing the linear equations for Gurobi using Function \ref{F:8.1} 
for $G = J_1$ and the maximal
subgroups isomorphic to $C^3_2 : C_7 : C_3$ for the elements of type $7A$, the Gurobi output shows that the number of subgroups of this type needed to
cover the respective elements is between $629$ and $752$.  (We have included the files produced by GAP which are read by Gurobi on \url{http://www.math.binghamton.edu/menger/coverings/}.)
Therefore, we find that the subgroup covering number of $J_1$ is between $1540 + 2926 + 629 + 186 = 5281$ and $1540 + 2926 + 752 + 196 = 5414$.
\end{proof}

\section{GAP Code}
\label{S:8}

In this section, we start with the code used in GAP \cite{Ga} to create the output files read by Gurobi \cite{Gu}.  Any solution to the system of 
equations encoded in the output corresponds to a subgroup cover of the elements, and any time the ``best objective" and the ``best bound" found by 
Gurobi are identical, Gurobi has found a minimal subgroup cover.  In short, GAP is used to create a system of linear inequalities, the optimal solution to which corresponds to a minimal cover. Gurobi then performs a linear optimization on this system of linear inequalities.

For the case of $S_9$, addressed in Proposition \ref{P:3.2}, we include the output of Function \ref{F:8.1} as well as an abbreviated table of the 
Gurobi output.
A complete table of this output can be found at \url{http://www.math.binghamton.edu/menger/coverings/}.  The corresponding output of Function \ref{F:8.1}, together with the generators for the subgroups in the minimal cover of elements with cycle structure $(6,6)$ in $M_{12}$ and the linear programs produced for $J_1$, can be found at the same website.

\begin{function}
\label{F:8.1}
GAP function to create the output files to be read by Gurobi.
\begin{footnotesize}
\begin{verbatim}
#SubgroupCoveringNumber  takes as input a group $G$, a list of 
#elements $L$, a list of maximal subgroups $M$, and the name of a 
#file of type .lp to which output is written.

SubgroupCoveringNumber:= function(G, ElementList, 
MaximalSubgroupList, filename)

local maxs, maxconjs, x, y, temp, elts, eltconjs, output, 
NumberSubgroups, NumberElements, i, j, FilteredSubgroupIndices;

#Subgroup covering number first computes all conjugate subgroups 
#of those in the list MaximalSubgroupList.

maxs:= [];

for x in MaximalSubgroupList do
	maxconjs:= ConjugateSubgroups(G,x);
	for y in maxconjs do
		Add(maxs, y);
	od;
od;

NumberSubgroups:= Length(maxs);

#All cyclic subgroups generated by the conjugates of the elements
#in ElementList are stored in the irredundant list elts.

elts:= [];

for x in ElementList do
	eltconjs:= AsList(ConjugacyClass(G,x));
	for y in eltconjs do
		if not Group(y) in elts	then
			Add(elts, Group(y));
		fi;
	od;
od;

NumberElements:= Length(elts);

#SubgroupCoveringNumber now begins writing to the output file.
#Each variable r1, r2,... represents a binary variable that takes 
#on the value 0 or 1. (A 1 represents the subgroup being included
# in the covering; a 0 means it's not included.)  

#First, we write that we want to minimize the sum of all the 
#variables, i.e., we want to minimize the number of subgroups
#included in the covering.

output := OutputTextFile( filename, false );;
  SetPrintFormattingStatus(output, false);
  AppendTo(output,"Minimize\n");

for i in [1..NumberSubgroups] do
      AppendTo(output, Concatenation( " + r", String(i)));
  od;
  AppendTo(output,"\n Subject To\n");

#For each subgroup H in elts, we require that H is a subgroup
#of at least one maximal subgroup in the covering.  This 
#corresponds to the sum over all the variables representing
#maximal subgroups containing H being at least 1.  Note that
#Gurobi interprets > as ``less than or equal."


for i in [1..NumberElements] do
	FilteredSubgroupIndices:= Filtered([1..NumberSubgroups], 
				        j -> (IsSubgroup(maxs[j],elts[i])));
	for j in FilteredSubgroupIndices do
		AppendTo(output, " + r", String(j));
	od;
	AppendTo(output, " > 1\n");
od;

#This last part specifies that each variable is ``Binary," i.e., that
#it can only take on the value 0 or the value 1.

AppendTo(output, "\\ Variables\n");
  AppendTo(output,"Binary\n");
  for i in [1..NumberSubgroups] do
      AppendTo(output, Concatenation( "r", String(i), "\n"));
  od;
  
  AppendTo(output,"End\n");
  CloseStream(output);
  return maxs;

#The function returns the list of maximal subgroups.

end;\end{verbatim}
\end{footnotesize}
\end{function}

As a sample of the output of Function \ref{F:8.1} we will show how the calculations proceed for the elements with cycle structure $(3,6)$ in the 
group $S_9$.  First, we use GAP to create a file that is readable by the optimization software Gurobi:
\begin{footnotesize}
\begin{verbatim}
gap> G:= SymmetricGroup(9);
Sym( [ 1 .. 9 ] )
gap> max:= MaximalSubgroupClassReps(G);
[ Alt( [ 1 .. 9 ] ), Group([ (1,2,3,4,5), (1,2), (6,7,8,9), (6,7) ]), 
  Group([ (1,2,3,4,5,6), (1,2), (7,8,9), (7,8) ]), 
  Group([ (1,2,3,4,5,6,7), (1,2), (8,9) ]), 
  Group([ (1,2,3,4,5,6,7,8), (1,2) ]), 
  Group([ (1,2,3), (1,2), (4,5,6), (4,5), (7,8,9), (7,8), 
      (1,4,7)(2,5,8)(3,6,9), (1,4)(2,5)(3,6) ]), 
  Group([ (4,7)(5,8)(6,9), (2,7,6)(3,4,8), (1,2,3)(4,5,6)(7,8,9) ]) ]
gap> M:= [max[3], max[6], max[7]];
[ Group([ (1,2,3,4,5,6), (1,2), (7,8,9), (7,8) ]), 
  Group([ (1,2,3), (1,2), (4,5,6), (4,5), (7,8,9), (7,8), 
      (1,4,7)(2,5,8)(3,6,9), (1,4)(2,5)(3,6) ]), 
  Group([ (4,7)(5,8)(6,9), (2,7,6)(3,4,8), (1,2,3)(4,5,6)(7,8,9) ]) ]
gap> g:= (1,2,3)(4,5,6,7,8,9);
(1,2,3)(4,5,6,7,8,9)
gap> L:= [g];
[ (1,2,3)(4,5,6,7,8,9) ]
gap> Read("Programs/SubgroupCoveringNumber.g");
gap> l:= SubgroupCoveringNumber(G,L,M, "S9.lp");;
gap> time;
218128
\end{verbatim}
\end{footnotesize}
Note that only one element with cycle structure $(3,6)$ is needed in the list $L$ since all elements with the same cycle structure are conjugate in a symmetric group.  We next use Gurobi to optimize this system of linear equations.  We have removed some lines of the output here for the sake of brevity, although the full output is available online at\\ \url{http://www.math.binghamton.edu/menger/coverings/}.

\begin{footnotesize}
\begin{verbatim}
gurobi> m = read("S9.lp")
Read LP format model from file S9.lp
Reading time = 0.09 seconds
(null): 10080 rows, 1204 columns, 80640 nonzeros
gurobi> m.optimize()
Optimize a model with 10080 rows, 1204 columns and 80640 nonzeros
Found heuristic solution: objective 423
Presolve time: 0.10s
Presolved: 10080 rows, 1204 columns, 80640 nonzeros
Variable types: 0 continuous, 1204 integer (1204 binary)

Root relaxation: objective 7.000000e+01, 2182 iterations, 0.19 seconds

    Nodes    |    Current Node    |     Objective Bounds      |     Work
 Expl Unexpl |  Obj  Depth IntInf | Incumbent    BestBd   Gap | It/Node Time

     0     0   70.00000    0  280  423.00000   70.00000  83.5%     -    0s
H    0     0                     180.0000000   70.00000  61.1%     -    0s
H    0     0                     123.0000000   70.00000  43.1%     -    0s
H    0     0                      84.0000000   70.00000  16.7%     -    0s
     0     0   70.65138    0  285   84.00000   70.65138  15.9%     -    1s
     0     0   70.78547    0  289   84.00000   70.78547  15.7%     -   14s
     0     0   70.98212    0  290   84.00000   70.98212  15.5%     -   28s
\end{verbatim}
\end{footnotesize}     

\begin{center}
      $\vdots$
\end{center}

\begin{footnotesize}
\begin{verbatim}  
   104    36   79.66667   12  272   84.00000   77.51814  7.72%  1228  389s
   345    56   82.42501    8  376   84.00000   78.08711  7.04%   839  412s
   490    28   80.48307    7  349   84.00000   78.08715  7.04%   761  430s
   606     3     cutoff    7        84.00000   78.08715  7.04%   724  441s
   698     2   79.66022    9  315   84.00000   79.66022  5.17%   688  450s

Cutting planes:
  Zero half: 107

Explored 717 nodes (514185 simplex iterations) in 453.03 seconds
Thread count was 8 (of 8 available processors)

Optimal solution found (tolerance 1.00e-04)
Best objective 8.400000000000e+01, best bound 8.400000000000e+01, gap 0.0%
\end{verbatim}
\end{footnotesize}

The ``Best objective" is the best actual solution that was found by Gurobi, and the size of this solution is 84.  The ``best bound" is the size of the best lower bound that Gurobi could determine for a solution to this system of equations, and this lower bound is also 84.  Therefore, we conclude that the covering of the elements with cycle structure $(3,6)$ in $S_9$ by the $84$ subgroups in MS3 is minimal. 

%\begin{acknowledgement}
%  ...
%\end{acknowledgement}

%\bibliographystyle{...}
%\bibliography{...}

\end{document}